\documentclass[11pt,twoside,a4paper]{article}
\usepackage{amsmath,amssymb,amsthm}
\usepackage{hyperref}

\usepackage{graphicx,psfrag}
\newtheorem{thm}{Theorem}[section]
\newtheorem{lem}[thm]{Lemma}
\newtheorem{pro}[thm]{Proposition}

\theoremstyle{definition}

\theoremstyle{remark}
\newtheorem{rem}[thm]{Remark}

\newcommand{\K}{\mathbb{K}}
\newcommand{\qH}{\mathbb{H}}
\newcommand{\R}{\mathbb{R}}

\newcommand{\C}{\mathbb{C}}

\newcommand{\Ca}{\C a}

\newcommand{\cB}{\mathcal{B}}
\newcommand{\cN}{\mathcal{N}}
\newcommand{\cR}{\mathcal{R}}

\newcommand{\al}{\alpha}
\newcommand{\be}{\beta}
\newcommand{\ga}{\gamma}

\newcommand{\de}{\delta}

\newcommand{\ep}{\varepsilon}
\newcommand{\om}{\omega}

\newcommand{\si}{\sigma}

\newcommand{\la}{\lambda}

\renewcommand{\phi}{\varphi}

\newcommand{\CAT}{\operatorname{CAT}}
\newcommand{\diam}{\operatorname{diam}}
\newcommand{\dist}{\operatorname{dist}}
\newcommand{\const}{\operatorname{const}}
\newcommand{\iso}{\operatorname{Isom}}
\newcommand{\hyp}{\operatorname{H}}

\newcommand{\grad}{\operatorname{grad}}
\newcommand{\im}{\operatorname{Im}}

\renewcommand{\d}{\partial}
\newcommand{\di}{\d_{\infty}}

\newcommand{\sm}{\setminus}
\newcommand{\sub}{\subset}

\newcommand{\ov}{\overline}
\newcommand{\wt}{\widetilde}

\begin{document}

\title{Boundary at infinity of symmetric rank one spaces}
\author{S.~Buyalo\footnote{Supported by RFFI Grant 08-01-00079a}
\ and A. Kuznetsov}

\date{}
\maketitle

\begin{abstract} We show that canonical Carnot-Carath\'eodory
spherical and horospherical metrics, which are defined on the boundary
at infinity of every rank one symmetric space of non-compact type,
are visual, i.e., they are bilipschitz equivalent with universal
bilipschitz constants to the inverse exponent of Gromov products
based in the space and on the boundary at infinity respectively.
\end{abstract}

\section{Introduction}

Let
$M$
be a rank one symmetric space of non-compact type. There
are two classes of natural metrics on the boundary
at infinity
$\di M$.
First, these are spherical and horospherical metrics. Given
$o\in M$,
we consider for every
$t>0$
the Riemannian metric
$ds_t^2$
induced from
$M$
on the sphere
$S_t\sub M$
of radius
$t$
centered at
$o$.
Identifying
$S_t$, $\di M$
with the unit sphere
$U_oM\sub T_oM$
via the radial projection from
$o$,
we consider
$ds_t^2$
as the Riemannian metric on
$U_oM$
for all
$t>0$.
Then there exists a limit
$$ds_\infty=\lim_{t\to\infty}e^{-t}ds_t,$$
which is a Carnot-Carath\'eodory metric. The respective
(bounded) interior distance is the {\em spherical} metric
$d_\infty$
on
$\di M$
based at
$o$.

Similarly, given a Busemann function
$b:M\to\R$
centered at
$\om\in\di M$,
we consider for every
$t>0$
the Riemannian metric
$ds_{b,t}^2$
induced from
$M$
on the horosphere
$H_{b,t}=b^{-1}(t)$.
Identifying
$H_{b,t}$, $\di M\sm\{\om\}$
with fixed horosphere
$H=H_{b,0}$
via the radial projection from
$\om$,
we consider
$ds_{b,t}^2$
as the Riemannian metric on
$H$
for all
$t\in\R$.
Then there exists a limit
$$ds_b=\lim_{t\to\infty}e^{-t}ds_{b,t},$$
which is a Carnot-Carath\'eodory metric. The respective
(unbounded) interior distance is the {\em horospherical} metric
$d_b$
on
$\di M\sm\{\om\}$
associated with
$b$.

Second, since
$M$
is a
$\CAT(-1)$-space,
in particular,
$M$
is Gromov hyperbolic, there are visibility metrics on the Gromov
boundary at infinity of
$M$.
The last coincides with the geodesic boundary
$\di M$,
and for
$\xi$, $\eta\in\di M$,
we have the Gromov product
$(\xi|\eta)_o$
based at
$o$
and similarly for
$\xi$, $\eta\in\di M\sm\{\om\}$,
we have the Gromov product
$(\xi|\eta)_b$
associated with a Busemann function
$b$.
Any (bounded) metric on
$\di M$,
which is bilipschitz equivalent to the function
$(\xi,\eta)\to a^{-(\xi|\eta)_o}$,
and any (unbounded) metric on
$\di M\sm\{\om\}$,
which is bilipschitz equivalent to the function
$(\xi,\eta)\to a^{-(\xi|\eta)_b}$
for some
$a>1$,
is called a {\em visibility} metric.

These two classes of metrics are certainly well known, but surprisingly,
we did not find in literature an answer to the natural
question how are they related (except for the case
$M=\hyp^n$
is real hyperbolic). The aim of the paper is to fill in this gap.
We obtain the following.

\begin{thm}\label{thm:main1} There are constants
$c_1$, $c_2>0$
such that for every symmetric rank one space
$M$
of non-compact type, the horospherical metric
$d_b$
on
$\di M\sm\{\om\}$,
associated with every Busemann function
$b:M\to\R$
centered at an arbitrary point
$\om\in\di M$,
satisfies
$$c_1e^{-(\xi|\eta)_b}\le d_b(\xi,\eta)\le c_2e^{-(\xi|\eta)_b}$$
for each
$\xi$, $\eta\in\di M\sm\{\om\}$.
\end{thm}

\begin{thm}\label{thm:main2} There are constants
$c_1$, $c_2>0$
such that for every symmetric rank one space
$M$
of non-compact type, the spherical metric
$d_\infty$
on
$\di M$
based at any
$o\in M$
satisfies
$$c_1e^{-(\xi|\eta)_o}\le d_\infty(\xi,\eta)\le c_2e^{-(\xi|\eta)_o}$$
for each
$\xi$, $\eta\in\di M$.
\end{thm}

\begin{rem}\label{rem:real_case} In the case
$M=\hyp^n$, $n\ge 2$,
is real hyperbolic, the metric
$2d_\infty$
is isometric to the standard metric of the unit sphere
$S^{n-1}\sub\R^n$
for every base point
$o\in\hyp^n$.
On the other hand, the function
$(\xi,\eta)\mapsto e^{-(\xi|\eta)_o}$
coincides with half of the chordal metric,
$$e^{-(\xi|\eta)_o}=\frac{1}{2}|\xi-\eta|$$
for every
$\xi$, $\eta\in S^{n-1}$
as an easy argument shows (see \cite[sect.~2.4.3]{BS}).

Similarly, for every Busemann function
$b:\hyp^n\to\R$
centered at
$\om\in\di\hyp^n$,
the horospherical metric
$d_b$
is isometric to the standard
$\R^{n-1}$,
and an easy calculation in the upper half-space model shows that
$d_b(\xi,\eta)=e^{-(\xi|\eta)_b}$
for every
$\xi$, $\eta\in\di\hyp^n\sm\{\om\}$.
\end{rem}

\begin{rem}\label{rem:complain_case} For the case
$M=\C\hyp^2$
is the complex hyperbolic plane, a weaker version of
Theorem~\ref{thm:main2} is obtained in \cite{Ku}.
\end{rem}

\begin{rem}\label{rem:metric} It is proved in \cite{Bo}
that the function
$(u,u')\mapsto |uu'|:=e^{-(u|u')_o}$
is a metric on the boundary at infinity
$\di X$
of any
$\CAT(-1)$-space
$X$
for every
$o\in X$.
Moreover, it is proved in \cite{FS} that
$\di X$
endowed with this metric is a {\em Ptolomy} metric space,
i.e., the Ptolomy inequality
$|xy|\cdot|uv|\le|xu|\cdot|yv|+|xv|\cdot|yu|$
holds for all
$x$, $y$, $u$, $v\in\di X$.
\end{rem}

\section{Preliminaries}

\subsection{Symmetric rank one spaces}

Every rank one symmetric space
$M$
of non-compact type is a hyperbolic space
$\K\hyp^n$
over the real numbers
$\K=\R$,
the complex numbers
$\K=\C$,
the quaternions
$\K=\qH$,
the octonions
$\K=\Ca$,
and
$\dim M=n\cdot\dim\K$,
where
$n\ge 2$
and
$n=2$
in the case
$\K=\Ca$
(see e.g. \cite{Wo}).
We use the standard notation
$TM$
for the tangent bundle of
$M$
and
$UM$
for the subbundle of the unit vectors.

\subsubsection{Curvature operator}

One of equivalent characterizations of Riemannian symmetric
spaces is that their curvature tensor is parallel,
$\nabla\cR=0$,
see e.g. \cite{Wo}. It follows that for every unit vector
$u\in U_oM$,
where
$o\in M$,
the eigenspaces
$E_u(\la)$
of the {\em curvature operator}
$\cR(\cdot,u)u:u^\bot\to u^\bot$,
where
$u^\bot\sub T_oM$
is the subspace orthogonal to
$u$,
are parallel along the geodesic
$\ga(t)=\exp_o(tu)$, $t\in\R$, and the respective
eigenvalues
$\la$
are constant along
$\ga$.
Note that
$u^\bot$
is identified with the tangent space
$T_u(U_oM)$
of the unit sphere
$U_oM$
at
$u$.

If
$M=\K\hyp^n$, $n\ge 2$
($n=2$
in the case of octonions), then the eigenvalues of curvature operator
$\cR(\cdot,u)u$
are
$\la=-1,-4$.
The dimensions of the respective eigenspaces are
$\dim E_u(-1)=(n-1)\dim\K$, $\dim E_u(-4)=\dim\K-1$,
$u^\bot=E_u(-1)\oplus E_u(-4)$.
In view of Remark~\ref{rem:real_case}, we assume that
$\K\neq\R$.
Then
$E_u(-4)\neq\{0\}$
and the sectional curvatures of
$M$
are pinched,
$-4\le K_\si\le -1$.
We denote by
$E(\la)$
the subbundle of the tangent bundle
$T(U_oM)$
with fibers
$E_u(\la)$, $u\in U_oM$, $\la=-1,-4$.

\subsection{Boundary at infinity and Carnot-Carath\'eodory metrics}

The (geodesic) boundary at infinity
$\di M$
of
$M$
consists of the equivalence classes of geodesic rays in
$M$,
where two rays are equivalent if they are at a finite
Hausdorff distance from each other, and it is equipped
with the standard (or cone) topology. In this topology,
the map
$r_o:U_oM\to\di M$,
which associates to
$u\in U_oM$
the class of the geodesic ray
$\exp_o(tu)$, $t\ge 0$,
is homeomorphism for every
$o\in M$.

\subsubsection{Busemann functions and horospheres}

The Busemann function
$b=b_\ga:M\to\R$
associated with the geodesic ray
$\ga\in\om\in\di M$,{
is defined by
$b(x)=\lim_{t\to\infty}(|x\ga(t)|-t)$.
In this case, we say that
$b$
is centered at
$\om$.
We denote by
$\cB(\om)$
the set of the Busemann functions centered at
$\om$.
Two functions
$b$, $b'\in\cB(\om)$
differ by a constant and, moreover,
$\cB(\om)$
is parametrized by the values of
$b\in\cB(\om)$
at a fixed point
$o\in M$,
which might be arbitrary.

For
$b\in\cB(\om)$, $t\in\R$,
the {\em horosphere}
$H_{b,t}=b^{-1}(t)$
is a smooth hypersurface in
$M$.
We denote again by
$E(\la)$
the subbundle of the tangent bundle
$TH_{b,t}$,
the fiber of which at every point
$x\in H_{b,t}$
is the eigenspace
$E_u(\la)$
of the curvature operator
$\cR(\cdot,u)u$,
where
$u=\grad b(x)$, $\la=-1,-4$.
We denote by
$r_{b,t}:H_{b,t}\to\di M\sm\{\om\}$
the radial projection map which associates to every
$x\in H_{b,t}$
the class
$\ga(\infty)\in\di M$,
where
$\ga:\R\to M$
is the (unique) geodesic with
$\ga(-\infty)=\om$
passing through
$x$.
Then
$r_{b,t}$
is a homeomorphism.

We denote by
$\rho_o:M\sm\{o\}\cup\di M\to U_oM$
the {\em radial projection},
$\rho_o(x)$
is the unit vector tangent to the geodesic
$ox$
at
$o$.

\begin{lem}\label{lem:horo_to_sphere} Given
$o\in M$, $\om\in\di M$, $b\in\cB(\om)$,
the composition
$f=\rho_o\circ r_{b,0}:H=H_{b,0}\to U_oM$
is a Lipschitz embedding, its differential
$df$
is well defined on the
$E(-1)$-subbundle of
$TH$,
preserves the
$E(-1)$-subbundles of the tangent bundles
$TH$, $T(U_oM)$
and is conformal on them.
\end{lem}

\begin{proof} The composition
$f_s:H\to U_oM$
of the radial projection
$H\to H_{b,s}$
from
$\om$
with
$\rho_o$
is a smooth embedding for all sufficiently large
$s\in\R$
that approximates
$f$
as
$s\to\infty$.
We show that the differential
$d_xf_s$
is uniformly bounded in
$s$
for every
$x\in H$
and almost preserves the
$E(-1)$-subbundles
as
$s\to\infty$.

Given
$x\in H$
and
$v\in E_u(\la)\sub T_xH$, $u=\grad b(x)$,
there is a unique Jacobi vector field
$V$
along the geodesic
$\ga:\R\to M$,
$\ga(-\infty)=\om$, $\ga(0)=x$,
such that
$V(0)=v$
and
$V(s)\to 0$
as
$s\to -\infty$.
The direction field
$V/|V|$
is parallel along
$\ga$
and
$|V(s)|=e^{s\sqrt{|\la|}}|v|$.

For
$s\in\R$
we put
$\al_s=\angle_{\ga(s)}(o,\om)$,
$\be_s=\angle_o(\om,\ga(s))$,
$\tau_s=|o\ga(s)|$.
Then
\begin{equation}\label{eq:angle_par}
\al_s\le4e^{-\tau_s}/\be_s
\end{equation}
by comparison with
$\hyp^2$
and the (generalized) formula for the angle of
parallelism. Note that
$\be_s\to\angle_o(\om,\ga(\infty))\neq 0$
as
$s\to\infty$
for all
$x\in H$.
Furthemore,
$\lim_{s\to\infty}(\tau_s-s)=b_{\ga_x}(o)$,
where
$\ga_x=x\ga(\infty)$.
Thus
$\tau_s=s+b_{\ga_x}(o)+o(1)$
as
$s\to\infty$.

For
$w=df_s(v)\in T(U_oM)$
let
$W$
be the Jacobi field along
$o\ga(s)$
with initial data
$W(o)=0$, $\dot W(o)=w$.
Then, by definition of
$w$,
the vector
$W(\ga(s))$
is the orthogonal projection of
$V(s)$
to
$u'^\bot$,
where
$u'\in U_{\ga(s)}M$
is tangent to the segment
$o\ga(s)$.
The angle between
$V(s)$
and
$W(\ga(s))$
is at most
$\al_s$.
Thus
$(1-\al_s^2)|V(s)|\le|W(\ga(s))|\le|V(s)|$.

In the orthogonal decomposition
$W=W_1+W_2$
with respect to the eigenspaces of the curvature
operator, the sections
$W_i$
of
$E(-i^2)$,
$i=1,2$,
are Jacobi fields along
$o\ga(s)$.
Furthemore, we have
$|W_i(\ga(s))|\le|W(\ga(s))|\le|V(s)|$.
The eigenspaces of the curvature operator
$\cR(\cdot,u)u$
depend smoothly on the direction
$u$,
thus for
$i=\sqrt{|\la|}$
the angle between
$W(\ga(s))$
and
$W_i(\ga(s))$
is
$O(\al_s)$,
and therefore
$|W_i(\ga(s))|\ge(1-O(\al_s))|V(s)|$.

Since
$\dot W=\dot W_1+\dot W_2$,
we have
$|W_i(\ga(s))|=\frac{\sinh\left(i\tau_s\right)}{i}|w_i|$
for
$i=1,2$
and the orthogonal decomposition
$w=w_1+w_2$
with respect to the eigenspaces of the curvature operator. Thus for
$i=\sqrt{|\la|}$,
we have
$$(1-O(\al_s))i|V(s)|/\sinh\left(i\tau_s\right)
\le|w_i|\le i|V(s)|/\sinh\left(i\tau_s\right)$$
and therefore
$|w_i|\sim 2ie^{-ib_{\ga_x}(o)}|v|$
as
$s\to\infty$.
Note that
$b_{\ga_x}(o)\ge -\dist(o,H)$
and
$b_{\ga_x}(o)\to\infty$
as
$x\to\infty$
in
$H$.

If
$\la=-1$,
then
$|w_2|\le 2|V(s)|/\sinh2\tau_s\le 8e^{-s}e^{-2b_{\ga_x}(o)}|v|\to 0$
as
$s\to\infty$.
This shows that
$df_s$
almost preserves the subbundle
$E(-1)$
and is almost conformal on it with the factor
$2e^{-b_{\ga_x}(o)}$.

It remains to consider the case
$\la=-4$.
We already know that
$|w_2|\sim 4e^{-2b_{\ga_x}(o)}|v|$
as
$s\to\infty$.
On the other hand,
$|W_1(\ga(s))|=O(\al_s)|V(s)|$
and thus
$|w_1|=|W_1(\ga(s))|/\sinh\tau_s\le O(\al_s)e^{2s}|v|/\sinh\tau_s$.
Using (\ref{eq:angle_par}), we obtain
$|w_1|\le ce^{-2b_{\ga_x}(o)}|v|$
for some constant
$c>0$
depending only on
$x\in H$, $c=c(x)$.
Thus
$|d_xf_s(v)|\le|w_1|+|w_2|\le(c+4)e^{-2b_{\ga_x}(o)}|v|$
and hence the norm
$\|d_xf_s\|$
is uniformly bounded in
$s$
for every
$x\in H$.
We conclude that
$f=\lim_{s\to\infty} f_s$
is Lipschitz.
\end{proof}

\begin{rem}\label{rem:pansu} Lemma~\ref{lem:horo_to_sphere}
is a refinement of \cite[Lemme~9.6]{Pa}. We have added the
estimate of
$df_s$
on
$E(-4)$
based on the estimate (\ref{eq:angle_par}), which leads
to the conclusion that the embedding
$f:H\to U_oM$
is Lipschitz.
\end{rem}

\subsubsection{Isometries of $M$ and invariant metrics on a horosphere}
\label{subsubsect:iso_metric_horo}

Given
$\om\in\di M$,
there is a subgroup
$N_\om$
in the isometry group
$G=\iso M$
(a maximal unipotent subgroup) that leaves invariant
every horosphere
$H$
in
$M$
centered at
$\om$
and is simply transitive on
$H$
and
$\di M\sm\{\om\}$.
The group
$N=N_\om$
is a nilpotent Lie group of dimension
$\dim N=\dim M-1$.
In the case
$M=\C\hyp^2$,
$N$
is isomorphic to the classical Heisenberg group.

Fixing a base point
$o\in H$,
we identify
$N$
with the orbit
$H=N(o)$
and the tangent space
$T_oH$
with the Lie algebra
$\cN$
of
$N$.
Then
$N$
leaves invariant the subbundles
$E(\la)$
of
$TH$
and
$\cN=E_1\oplus E_2$,
where
$E_i=E_u(-i^2)$, $i=1,2$,
are the fibers of
$E(\la)$'s
at
$o$.
Moreover,
$E_2=[E_1,E_1]$
in
$\cN$,
in particular,
$\cN$
is the minimal subalgebra in
$\cN$
containing
$E_1$,
see e.g. \cite{Pa}.

The subbundle
$E(-4)\sub TH$
is intergrable and respective fibers
$F$
are intersections of
$\K$-lines
with
$H$,
where each
$\K$-line
in
$M$
is a totally geodesic subspace isometric to the
hyperbolic space
$\frac{1}{2}\hyp^{\dim\K}$
of constant curvature
$-4$,
see \cite[\S 20]{Mo}, \cite{Pa}. Therefore, every fiber
$F$
is flat in
$H$.
The corresponding fibration
$\phi:H\to B$
is a Riemannian submersion (with respect to the induced
Riemannian metric on
$H$)
with the base
$B$
isometric to
$\R^{(n-1)\dim\K}$.
The group
$N$
acts by isometries on
$B$,
the kernel of that action is the center
$Z$
of
$N$,
so that
$N/Z=\R^{(n-1)\dim\K}$,
and the submersion
$\phi$
is equivariant with respect to these actions of
$N$
on
$H$
and
$B$.

Identifying
$E_1$
with
$\K^{n-1}$
and
$E_2$
with
$\im K$,
the Lie algebra structure of
$\cN=E_1\oplus E_2$
is given by
$[v,w]=2\im(v,w)$
for
$v$, $w\in E_1$,
where
$v=(v_2,\dots,v_n)$
and
$(v,w)=\sum_i\ov v_iw_i$,
see \cite[(19.14)]{Mo}.

We call the subbundle
$E=E(-1)\sub TH$
the {\em polarization} on the horosphere
$H$.
A piecewise smooth curve
$\si:I\to H$
is said to be
$E$-{\em horizontal},
if
$\dot\si(t)\in E$
for every
$t\in I\sub\R$.
Its length is
$\ell(\si)=\int_I|\dot\si(t)|dt$.
We define the distance
$d_E(x,x')$
by taking the infimum of lengths of
$E$-horizontal
curves between
$x$, $x'\in H$.
This is finite because
$E_1$
generates
$\cN$
and thus
$E$
is completely non-intergrable. We call the restriction
$ds_E^2=ds^2|E$
of the Riemannian metric
$ds^2$
of
$M$
the {\em Carnot-Carath\'eodory} metric and respective distance
$d_E$
the {\em Carnot-Carath\'eodory} distance on
$H$.
Furthemore,
$d\phi$
is isometric on
$E$
and thus
$\phi$
preserves the lengths of
$E$-horizontal
curves.

Denote by
$d_H$
the interior distance on
$H$
induced from
$M$.

\begin{lem}\label{lem:basic} We have
$$\sup\frac{d_E^2(o,x)}{d_H(o,x)}\le 17,$$
where the supremum is taken over all
$x\in H$
with
$|ox|\le 1$
in
$M$.
\end{lem}

\begin{proof} Let
$F=\phi^{-1}(\phi(o))$
be the fiber of
$\phi$
through
$o$.
First, we show that the ratio
$r(x)=d_E^2(o,x)/d_H(o,x)$
is independent of
$x\in F$.

For every
$\la>0$
there is a homothety
$h_\la:N\to N$
that is an automorphism of
$N$
and that acts on
$F$
as the homothety with coefficient
$\la^2$
and on
$B$
as the homothety with coefficient
$\la$,
see e.g. \cite{Pa}. In particular,
$\ell(h_\la(\si))=\la\ell(\si)$
for the length
$\ell(\si)$
of every
$E$-horizontal
curve
$\si\sub H$.
It follows that
$r(h_\la(x))=r(x)$
for every
$\la>0$.
Furthemore, the stabilizer
$G_{o\om}\sub G$
of
$o\cup\om$
acts transitively on the rays
$o\xi\sub F$
(see \cite[\S~21]{Mo}). Thus
$r=r(x)$
is independent of
$x\in F\sm\{o\}$.

Given
$v$, $w\in E_1\sub\cN=T_oH$,
the commutator
$[\exp sv,\exp sw]=x_s$
lies in
$F=\exp E_2$
for every
$s\ge 0$
and
$[v,w]=\lim_{s\to 0}\frac{1}{s^2}ox_s$.
Representing the commutator as the end point of
the broken geodesic
$\si_s$
with four edges, we find that
$\si_s$
is an
$E$-horizontal
curve of length
$2s(|v|+|w|)$
that connects
$o$
and
$x_s$.
Thus
$d_E(o,x_s)\le 2s(|v|+|w|)$,
while
$d_H(o,x_s)=s^2|[v,w]|+o(s^2)$.
Hence
$r\le 4(|v|+|w|)^2/|[v,w]|$.
Taking
$v$, $w\in E_1=\K^{n-1}$
orthonormal,
$v=(1,0,\dots,0)$, $w=(j,0,\dots,0)$,
where
$j$
is one of
$\dim\K-1$
imaginary units,
we find
$[v,w]=2j$.
Thus
$|v|=|w|=1$, $|[v,w]|=2$
and
$r\le 8$.

In general case, given
$x\in N$
with
$|ox|\le 1$
in
$M$,
there is an
$E$-horizontal
curve
$\si$
that connects
$o$
and
$x'$
so that
$\phi\circ\si$
is a segment in
$B$
and
$x'x\sub F$,
where
$F$
is the fiber of
$\phi$
through
$x$.
We have
$|x'x|\le|ox|\le 1$, $d_H(o,x)\ge\ell(\si),d_H(x',x)$
and
$\si$
lies in a totally real plane in
$M$
which is a geodesic subspace isometric to
$\hyp^2$.
Moreover,
$\si$
is a horocycle in
$\hyp^2$,
thus
$\ell(\si)\le 2\sinh\frac{1}{2}$,
and also
$d_H(x',x)\le\sinh 1$,
see Theorem~\ref{thm:horo_dist} below. Then
$d_E(o,x)\le\ell(\si)+d_E(x',x)$
and thus
$$\frac{d_E^2(o,x)}{d_H(o,x)}\le\frac{d_E^2(x',x)}{d_H(x',x)}
+\ell(\si)+2d_E(x',x)\le r+2\sinh(1/2)+2\sqrt{r\sinh 1}\le 17.$$
\end{proof}

\subsection{Negatively pinched Hadamard manifolds}
\label{subsect:hadamard}

Let
$M$
be a Hadamard manifold with sectional curvatures
$-b^2\le K_\si\le -a^2$, $a>0$.
Assume that a horosphere
$H\sub M$
is fixed. We denote by
$|xx'|$
the distance between
$x$, $x'\in H$
in
$M$
and by
$|xx'|_H$
the induced interior distance in
$H$.
We shall use the following result from
\cite[Theorem~4.6]{HI}.

\begin{thm}\label{thm:horo_dist} For every
$x$, $x'\in H$
we have
$$\frac{2}{a}\sinh\frac{a}{2}|xx'|
\le|xx'|_H\le\frac{2}{b}\sinh\frac{b}{2}|xx'|.$$
\end{thm}

\subsection{Gromov hyperbolic spaces}
\label{subsect:basic_hyperbolic}

Most of facts on hyperbolic spaces needed for the paper,
except for the equiradial points of infinite triangles,
see sect.~\ref{subsubsect:equirad}, the reader can find,
for example, in \cite{BS}. We discuss them for the symmetric rank one spaces
$M$,
though all of them hold true for general proper $\CAT(-1)$-spaces
$X$.
In this case the geodesic boundary at infinity
$\di X$
coincides with the Gromov boundary at infinity. Every
$\CAT(-1)$-space
is boundary continunous, see \cite[sect. 3.4.2]{BS}. This essentially
simplifies definitions of key notions comparing with general case of
Gromov hyperbolic spaces.

\subsubsection{Gromov products and $\de$-triples}

The {\em Gromov product} of
$x$, $x'\in M$
with respect to
$o\in M$
is
$(x|x')_o = 1/2(|xo|+|x'o|-|xx'|)$.
This is always nonnegative by triangle inequality.

The space
$M$
being a
$\CAT(-1)$-space
is
$\de$-hyperbolic (by Gromov),
that is, for every triangle
$xyz \sub M$
if
$y'\in xy$, $z'\in xz$
and
$|xy'| = |xz'|\le (y|z)_x$,
then
$|y'z'|\le\de$,
with the constant
$\de\le\de_{\hyp^2}$.
It is known that the hyperbolicity constant for the
real hyperbolic plane
$\hyp^2$
equals
$\de_{\hyp^2}=2\ln\tau=0.9624\dots$,
where
$\tau$
is the golden ratio,
$\tau^2=\tau+1$,
see \cite{BS}.

The condition above implies that for every
$o,x,y,z\in M$
the
$\de$-{\em inequality}
holds,
$(x|y)_o \geq \min\{(x|z)_o,(y|z)_o\} - \de$
(the converse is also true but with different
$\de$).
This inequality can be rewritten in terms of
$\de$-triples.
A triple of real numbers
$(a,b,c)$
is called a
$\de$-{\em triple},
if the two smallest of these numbers differ by at most
$\de$.
Then the
$\de$-inequality
is equivalent to that the numbers
$(x|y)_o$,
$(x|z)_o$
and
$(y|z)_o$
form a
$\de$-triple.

The Gromov product of
$x$, $y\in M$
with respect to a Busemann function
$b:M\to\R$
is defined by
$$(x|y)_b = \frac{1}{2}(b(x) + b(y) - |xy|).$$
This product, contrary to the standard case
$(x|y)_o$,
may take arbitrary real values. Busemann functions
$b$, $b'$,
centered at one and the same point
$\om\in\di M$,
differ by a constant,
$b-b'=\const$.
Thus the Gromov products with respect to
$b$, $b'$
differ by the same constant. As above, for every
$x$, $y$, $z\in M$
the numbers
$(x|y)_b$,
$(x|z)_b$
and
$(y|z)_b$
form a
$\de$-triple.

The Gromov product of points
$\xi$, $\eta\in\di M$
in the boundary at infinity with respect to
$o\in M$
is the limit
$$(\xi|\eta)_o=\lim_{s,t\to\infty}(\ga(s)|\rho(t))_o,$$
where
$\ga=o\xi$, $\rho=o\eta$
are geodesic rays (the limit exists by monotonicity of
the Gromov product). Note that
$(\xi|\xi)_o=\infty$
for every
$\xi\in\di M$.
Again, for every pairwise distinct points
$\xi$, $\eta$, $\zeta\in\di M$
the numbers
$(\xi|\eta)_o$,
$(\xi|\zeta)_o$
and
$(\eta|\zeta)_o$
form a
$\de$-triple.

Finally, the Gromov product of
$\xi$, $\eta\in\di M\sm\{\om\}$
with respect to a Busemann function
$b:M\to\R$
centered at
$\om\in\di M$
is the limit
$$(\xi|\eta)_b=\lim_{s,t\to\infty}(\ga(s)|\rho(t))_b,$$
where
$\ga=\om\xi$, $\rho=\om\eta$
are geodesics in
$M$
with
$\ga(\infty)=\xi$, $\rho(\infty)=\eta$
(the limit exists by monotonicity of the Gromov product).
For different Busemann functions
$b$, $b':M\to\R$
with one and the same center
$\om$
the Gromov products
$(\xi|\eta)_b$, $(\xi|\eta)_{b'}$
differ by a constant,
$(\xi|\eta)_b-(\xi|\eta)_{b'}=\const$,
for every
$\xi$, $\eta\in\di M\sm\{\om\}$.
For every pairwise distinct points
$\xi$, $\eta$, $\zeta\in\di M\sm\{\om\}$
the numbers
$(\xi|\eta)_b$,
$(\xi|\zeta)_b$
and
$(\eta|\zeta)_b$
form a
$\de$-triple.

\subsubsection{Equiradial points of a triangle}
\label{subsubsect:equirad}

For every triangle
$xyz\sub M$
there are three spheres centered at its vertices
which are pairwise tangent to each other from outside.
The tangent points lie on the sides of the triangle and are
called the {\em equiradial} points. The collection of the
equiradial points
$u\in yz$, $v\in xz$, $w \in xy$
is uniquely determined by the conditions
$|uz|=|vz|=(x|y)_z$, $|uy|=|wy|=(x|z)_y$, $|vx|=|wx|=(y|z)_x$.

If instead of spheres we take horospheres centered in vertices
at infinity, then we obtain the definition of equiradial points
of an infinite triangle in
$M$.
In this case, the existence of the equiradial points and their
connection with respective Gromov products is not completely
obvious.

\begin{pro}\label{pro:equirad_inf} For every pairwise distinct points
$\xi$, $\eta$, $\zeta\in\di M$
the infinite triangle
$\xi\eta\zeta$
in
$M$
possesses a unique collection of equiradial points
$u\in\eta\zeta$, $v\in\xi\zeta$, $w\in\xi\eta$.
Furthemore, we have
$(\eta|\zeta)_b=b(v)=b(w)$
for every
$b\in\cB(\xi)$,
$(\xi|\zeta)_{b'}=b'(u)=b'(w)$
for every
$b'\in\cB(\eta)$
and
$(\xi|\eta)_{b''}=b''(u)=b''(v)$
for every
$b''\in\cB(\zeta)$.
\end{pro}

\begin{proof} There are uniquely determined points
$v\in\xi\zeta$, $w\in\xi\eta$,
such that
$(\eta|\zeta)_b=b(v)=b(w)$
for every
$b\in\cB(\xi)$.
We take the function
$b$
so that
$(\eta|\zeta)_b=0$.
We fix a function
$b'\in\cB(\zeta)$
such that
$b'(v)=0$
and let
$u=\eta\zeta\cap b'^{-1}(0)$.
We fix a function
$b''\in\cB(\eta)$
such that
$b''(u)=0$
and let
$w'=\xi\eta\cap b''^{-1}(0)$.
We show that
$b(w')=0$,
i.e.,
$w'=w$.

Denote by
$\eta(t)\zeta(t)\sub\eta\zeta$
the segment of length
$2t$
centered at
$u$.
Let
$\eta_\xi:[0,\infty)\to\xi\eta$
be the natural parametrization of the ray
$w'\eta$, $\eta_\xi(0)=w'$, $\eta_\xi(\infty)=\eta$,
$\zeta_\xi:[0,\infty)\to\xi\zeta$
the natural parametrization of the ray
$v\zeta$, $\zeta_\xi(0)=v$, $\zeta_\xi(\infty)=\zeta$.
Since
$b'(u)=b'(v)=0=b(v)$
and
$b''(w')=b''(u)=0$,
we have
$b(\zeta_\xi(t))=t$, $b(\eta_\xi(t))=b(w')+t$
and
$||\eta_\xi(t)\zeta_\xi(t)|-|\eta(t)\zeta(t)||=o(1)$
as
$t\to\infty$,
and we obtain

\begin{eqnarray*}
0=(\eta|\zeta)_b=\lim_{t\to\infty}(\eta_\xi(t)|\zeta_\xi(t))_b
&=&\frac{1}{2}\lim_{t\to\infty}(b(\eta_\xi(t))+b(\zeta_\xi(t))
-|\eta_\xi(t)\zeta_\xi(t)|)\\
&=&\frac{1}{2}\lim_{t\to\infty}(b(\eta_\xi(t))-t+b(\zeta_\xi(t))-t+o(1))\\
&=&b(w')/2.
\end{eqnarray*}

This means that the horospheres
$b^{-1}(0)$, $b'^{-1}(0)$, $b''^{-1}(0)$
centered at
$\xi$, $\eta$, $\zeta$
respectively touch each other from outside, i.e., the points
$u$, $v$, $w$
are equiradial for the infinite triangle
$\xi\eta\zeta$.

Since
$(\eta|\zeta)_b=b(v)=b(w)$
and the equiradial points are determined uniquely,
the equalities
$(\xi|\eta)_{b'}=b'(u)=b'(v)$,
$(\xi|\zeta)_{b''}=b''(u)=b''(w)$
also hold.
\end{proof}

\begin{rem}\label{rem:equirad_fin} The equiradial points of
infinite triangles in
$M$
with one or two vertices in
$M$
also exist and uniquely determined, and for them
the respective equalities for the Gromov products
also hold. This can be proved by similar arguments.
\end{rem}

\begin{lem}\label{lem:equirad_dist} Let
$u\in\eta\zeta$, $v\in\xi\zeta$, $w\in\xi\eta$
be the equiradial points of an infinite triangle in
$M$
with pairwise distinct vertices
$\xi$, $\eta$, $\zeta\in\di M$.
Then
$|uv|$, $|uw|$, $|vw|\le\de$,
where
$\de>0$
is the hyperbolicity constant for
$M$.
\end{lem}

\begin{proof} The equiradial points of every finite
triangle in
$M$
are at the distance at most
$\de$
from each other
by definition of
$\de$-hyperbolicity.
Thus it suffices to find a family of finite
triangles in
$M$
the collections of equiradial points of which approximate
the collection
$u$, $v$, $w$.

For
$t>0$
consider
$x\in v\xi$, $y\in w\eta$, $z\in u\zeta$
at the distance
$t$
from the vertices of the rays,
$|xv|=|yw|=|zu|=t$.
As
$t\to\infty$
we have
$wx\to w\xi$, $uy\to u\eta$, $vz\to v\zeta$,
and each of distances
$|wx|$, $|uy|$, $|vz|$
differs from
$t$
by
$o(1)$.
Thus each of distances
$|yx|$, $|zy|$, $|xz|$
differs from
$2t$
by
$o(1)$.
Hence, the Gromov products
$(y|z)_x$, $(x|z)_y$, $(x|y)_z$
all are equal to
$t$
up to
$o(1)$.
This immediately implies that the equiradial points of
$xyz$
approximate the points
$u$, $v$, $w$
as
$t\to\infty$.
\end{proof}

\begin{lem}\label{lem:equirad_below} Let
$u\in\eta\zeta$, $v\in\xi\zeta$, $w\in\xi\eta$
be the equiradial points of an infinite triangle
$\xi\eta\zeta\sub M$
with pairwise distinct vertices
$\xi$, $\eta$, $\zeta\in\di M$.
Then for the points
$w'\in w\eta$, $v'\in v\zeta$
with
$|w'w|=1+\de=|v'v|$,
where
$\de$
is the hyperbolicity constant of
$M$,
we have
$|v'w'|\ge 2$.
\end{lem}

\begin{proof} We take
$w''\in u\eta$, $v''\in u\zeta$
so that
$|w''u|=1+\de=|v''u|$.
Then it follows from Lemma~\ref{lem:equirad_dist} that
$|v'v''|\le\de$
and
$|w'w''|\le\de$.
Thus
$|v'w'|\ge|v''w''|-2\de=2$.
\end{proof}

\section{Proof of Theorem~\ref{thm:main1}}
\label{sect:proof1}

We denote by
$\de$
the hyperbolicity constant of
$M$, $\frac{1}{2}\de_{\hyp^2}\le\de\le\de_{\hyp^2}<1$
by comparison with
$\hyp^2$, $\frac{1}{2}\hyp^2$.
We fix
$o\in M$, $\om\in\di M$
and the Busemann function
$b\in\cB(\om)$
with
$b(o)=0$.
For
$t\in\R$
we denote by
$H_{b,t}=b^{-1}(t)$
the horosphere in
$M$, $H=H_{b,0}$, $ds_{b,t}^2$
the Riemannian metric on
$H_{b,t}$
induced from
$M$.
We identify
$H_{b,t}$, $\di M\sm\{\om\}$
with
$H$
via the radial projection from
$\om$
and consider
$ds_{b,t}^2$
as the Riemannian metric on
$H$
for all
$t\in\R$.
Then the limit
$$ds_b=\lim_{t\to\infty}e^{-t}ds_{b,t}$$
exists and coincides with the Carnot-Carath\'eodory metric
$ds_E$.

This is because for every
$x\in H$, $v\in E_u(\la)\sub T_xH$,
where
$u=\grad b(x)$,
the unique Jacobi field
$V$
along the geodesic
$\exp_xtu$, $t\in\R$
with
$V(t)\to 0$
as
$t\to\infty$
has the parallel direction field
$V/|V|$
and
$|V(t)|=e^{t\sqrt{|\la|}}|v|$.
Thus
$e^{-t}ds_{b,t}(v)=e^{(\sqrt{|\la|}-1)t}|v|$,
and for
$v\in E=E(-1)$
we have
$e^{-t}ds_{b,t}(v)=|v|=ds_b(v)$
for all
$t\in\R$,
while
$ds_b(v)=\infty$
for all nonzero
$v\in E(-4)$.

We consider the respective Carnot-Carath\'eodory distance
$d_b$
as a metric on
$\di M\sm\{\om\}$
called {\em horospherical.}

\begin{lem}\label{lem:below} For each
$\xi$, $\eta\in\di M\sm\{\om\}$,
we have
$$d_b(\xi,\eta)\ge c_1e^{-(\xi|\eta)_b},$$
where
$c_1=2e^{-(1+\de)}$.
\end{lem}

\begin{proof} We assume that
$\xi\neq\eta$,
since otherwise there is nothing to prove. For
$\ep>0$
we take an
$E$-horizontal
curve
$\si_\infty\sub\di M\sm\{\om\}$
between
$\xi$
and
$\eta$
with length
$\ell_b(\si_\infty)\le d_b(\xi,\eta) + \ep$.
Let
$\si_t\sub H_{b,t}$
be its radial projection from
$\om$.
The curve
$\si_t$
connects the points
$u_t\in\om\eta$, $v_t\in\om\xi$
with
$b(u_t)=b(v_t)=t$
and also is
$E$-horizontal
on
$H_{b,t}$,
thus its length
$\ell_{b,t}(\si_t)=e^t\ell_b(\si_\infty)$.

Let
$u\in\om\eta$, $v\in\om\xi$
be the equiradial points of the triangle
$\om\xi\eta$.
Then
$t_0:=(\xi|\eta)_b=b(u)=b(v)$.
We put
$t=t_0+1+\de$.
For
$u_t\in\om\eta$, $v_t\in\om\xi$
we have
$|u_tv_t|\ge 2$
by Lemma~\ref{lem:equirad_below}. Thus
$\ell_{b,t}(\si_t)\ge|u_tv_t|\ge 2$
and
$$d_b(\xi,\eta)+\ep\ge\ell_b(\si_\infty)=e^{-t}\ell_{b,t}(\si_t)
\ge c_1e^{-(\xi|\eta)_b}.$$
Since
$\ep$
is taken arbitrary, we obtain the required estimate.
\end{proof}

\begin{proof}[Proof of Theorem~\ref{thm:main1}] In view of
Lemma~\ref{lem:below}, it remains to estimate the horospherical
distance
$d_b(\xi,\eta)$
from above. We assume that the points
$\xi$, $\eta\in\di M\sm\{\om\}$
are distinct. Let
$u\in\om\eta$, $v\in\om\xi$
be the equiradial points of the triangle
$\om\xi\eta$, $t=(\xi|\eta)_b=b(u)=b(v)$.
Then
$|uv|\le\de<1$
by Lemma~\ref{lem:equirad_dist}. Thus by Lemma~\ref{lem:basic},
we have
$d_E^2(u,v)\le 17d_{b,t}(u,v)$,
where
$d_{b,t}$
is the interior distance on
$H_{b,t}$
induced from
$M$.
By Theorem~\ref{thm:horo_dist}, we have
$d_{b,t}(u,v)\le\sinh|uv|\le\sinh\de$.
Thus
$d_E^2(u,v)\le 17\sinh\de=:c_2^2$.
To avoid a standard
$\ep$-argument,
we assume for simplicity that there is an
$E$-horizontal
curve
$\si_t\sub H_{b,t}$
between
$u$, $v$
with
$\ell_{b,t}(\si_t)=d_E(u,v)$.
The curve
$\si_t$
is the radial projection from
$\om$
of an
$E$-horizontal
curve
$\si_\infty\sub\di M\sm\{\om\}$
between
$\xi$, $\eta$,
and we have
$$d_b(\xi,\eta)\le\ell_b(\si_\infty)=e^{-t}\ell_{b,t}(\si_t)
=e^{-t}d_E(u,v)\le c_2e^{-(\xi|\eta)_b}.$$
\end{proof}

\section{Proof of Theorem~\ref{thm:main2}}
\label{sect:proof2}

We fix
$o\in M$
and consider for every
$t>0$
the Riemannian metric
$ds_t^2$
induced from
$M$
on the sphere
$S_t\sub M$
of radius
$t$
centered at
$o$.
Identifying
$S_t$, $\di M$
with the unit sphere
$U_oM\sub T_oM$
via the radial projection from
$o$,
we consider
$ds_t^2$
as the Riemannian metric on
$U_oM$
for all
$t>0$.
Then there exists a limit
$$ds_\infty=\lim_{t\to\infty}e^{-t}ds_t.$$
This is because given
$v\in E_u(\la)$,
where
$u\in U_oM$,
there is a unique Jacobi field
$V$
along the geodesic ray
$\ga(t)=\exp_otu$, $t\ge 0$,
with initial data
$V(0)=0$, $\dot V(0)=v$.
The direction field
$V/|V|$
is parallel along
$\ga$
and
$|V(t)|=\sinh(t\sqrt{|\la|})|v|/\sqrt{|\la|}$
for all
$t\ge 0$.
Thus
$e^{-t}ds_t(v)=e^{-t}|V(t)|\to |v|/2$
for
$\la=-1$
and
$\to\infty$
for
$\la=-4$.

We call the subbundle
$E=E(-1)\sub TU_oM$
the {\em polarization} on the sphere
$U_oM$.
A piecewise smooth curve
$\si:I\to U_oM$
is said to be
$E$-{\em horizontal},
if
$\dot\si(t)\in E$
for every
$t\in I\sub\R$.
Its length is
$\ell(\si)=\int_I|\dot\si(t)|dt$.
We define the distance
$d_E(u,u')$
by taking the infimum of lengths of
$E$-horizontal
curves between
$u$, $u'\in U_oM$.
This is finite because for every horosphere
$H\sub M$
the canonical embedding
$f:H\to U_oM$
is Lipschitz and its differential
$df$
preserves the polarizations on
$H$, $U_oM$
by Lemma~\ref{lem:horo_to_sphere}. We choose
$\om\in\di M$
so that neither
$u$
nor
$u'$
is tangent to
$o\om$,
take
$b\in\cB(\om)$
with
$b(o)=0$
and
$H=b^{-1}(0)$.
Then
$u$, $u'\in f(H)$
and there is an
$E$-horizontal
curve
$\si\sub H$
between
$x=f^{-1}(u)$, $x'=f^{-1}(u')$.
Then the curve
$f(\si)\sub U_oM$
is
$E$-horizontal,
connects
$u$, $u'$
and
$\ell_E(f(\si))\le 2\ell_E(\si)$
by the proof of Lemma~\ref{lem:horo_to_sphere}.

We denote by
$d_\infty$
the Carnot-Carath\'eodory distance on
$\di M$
associated with the Carnot-Carath\'eodory metric
$ds_\infty$.
Note that
$ds_\infty=\frac{1}{2}ds_E$
and
$d_\infty=\frac{1}{2}d_E$.
Thus the argument above shows that for every
$o\in M$, $\om\in\di M$, $b\in\cB(\om)$,
we have
$$ds_\infty\le ds_b \quad\text{and}\quad d_\infty\le d_b$$
on
$\di M\sm\{\om\}$.
The following lemma is a minor modification of
Lemma~\ref{lem:below}.

\begin{lem}\label{lem:below_sph} For each
$\xi$, $\eta\in\di M$,
we have
$$d_\infty(\xi,\eta)\ge c_1e^{-(\xi|\eta)_o},$$
where
$c_1=2e^{-(1+\de)}$.
\end{lem}

\begin{proof} We assume that
$\xi\neq\eta$,
since otherwise there is nothing to prove. For
$\ep>0$
we take an
$E$-horizontal
curve
$\si_\infty\sub\di M$
between
$\xi$
and
$\eta$
with length
$\ell_\infty(\si_\infty)\le d_\infty(\xi,\eta) + \ep$.
Let
$\si_t\sub S_t$
be its radial projection from
$o$.
The curve
$\si_t$
connects the points
$u_t\in o\eta$, $v_t\in o\xi$
with
$|ou_t|=|ov_t|=t$
and also is
$E$-horizontal
on
$S_t$,
thus its length
$\ell_t(\si_t)=2\sinh t\cdot\ell_\infty(\si_\infty)$.

Let
$u\in o\eta$, $v\in o\xi$
be the equiradial points of the triangle
$o\xi\eta$.
Then
$t_0:=(\xi|\eta)_o=|ou|=|ov|$.
We put
$t=t_0+1+\de$.
For
$u_t\in o\eta$, $v_t\in o\xi$
we have
$|u_tv_t|\ge 2$
by Lemma~\ref{lem:equirad_below} (applied to
$o\xi\eta$).
Thus
$\ell_t(\si_t)\ge|u_tv_t|\ge 2$
and
$$d_\infty(\xi,\eta)+\ep\ge\ell_\infty(\si_\infty)
=\frac{1}{2\sinh(t_0+1+\de)}\ell_t(\si_t)\ge c_1e^{-(\xi|\eta)_o}.$$
Since
$\ep$
is taken arbitrary, we obtain the required estimate.
\end{proof}

\begin{lem}\label{lem:distort} Assume that points
$\xi$, $\eta$, $\om\in\di M$
are pairwise distinct and points
$v$, $v'\in\xi\om$, $w$, $w'\in\xi\eta$
satisfy
$b(v)=b(w)$, $b(v')=b(w')$
for some and hence any Busemann function
$b\in\cB(\xi)$.
Then
$$|vv'|=|ww'|\le\ln\frac{\sinh A}{a},$$
where
$A=\max\{|vw|,|v'w'|\}$, $a=\min\{|vw|,|v'w'|\}$.
\end{lem}

\begin{proof} We put
$t=b(v)=b(w)$, $t'=b(v')=b(w')$
and without loss of generality assume that
$t'>t$.
Let
$\si_{t'}\sub H_{b,t'}$
be a shortest curve between
$v'$, $w'$,
i.e.
$\ell_{t'}(\si_{t'})=d_{b,t'}(v',w')$,
$\si_t\sub H_{b,t}$
the radial projection of
$\si_{t'}$
from
$\xi$.
Then
$e^{t-t'}d_{b,t}(v,w)\le e^{t-t'}\ell_t(\si_t)\le\ell_{t'}(\si_{t'})
=d_{b,t'}(v',w')$.
We have
$d_{b,t}(v,w)\ge a$
and by Theorem~\ref{thm:horo_dist},
$d_{b,t'}(v',w')\le \sinh A$.
Thus
$|vv'|=|ww'|=t'-t\le\ln\frac{\sinh A}{a}$.
\end{proof}

\begin{lem}\label{lem:grprodiff} Assume that
$(\xi|\eta)_o\ge 1+\de$
for some points
$o\in M$, $\xi$, $\eta\in\di M$.
Then for
$\om\in\di M$
opposite to
$\xi$
with respect to
$o$,
i.e.
$o\in\om\xi$,
and the function
$b\in\cB(\om)$, $b(o)=0$,
we have
$$|(\xi|\eta)_b-(\xi|\eta)_o|\le c_3=\ln\frac{\sinh(2+3\de)}{2}.$$
\end{lem}

\begin{proof} We can assume that
$\xi\neq\eta$,
since otherwise
$(\xi|\eta)_b=(\xi|\eta)_o=\infty$.
Furthemore,
$(\xi|\om)_o=0$,
thus
$\eta\neq\om$,
and
$\om$, $\xi$, $\eta$
are pairwise distinct.

Let
$v\in o\xi$, $u\in o\eta$, $w\in\xi\eta$
be the equiradial points of the triangle
$o\xi\eta$.
Then
$|ov|=(\xi|\eta)_o\ge 1+\de$
and thus there is
$v'\in ov$
such that
$|vv'|=1+\de$.
There also is
$w'\in w\eta$
such that
$|ww'|=1+\de$.
Then
$|v'w'|\le 2+3\de$.
On the other hand, as in Lemma~\ref{lem:equirad_below},
we obtain
$|v'w'|\ge 2$.

Applying a similar argument to the equiradial points
$\wt v\in\xi\om$, $\wt w\in\xi\eta$, $\wt u\in\om\eta$
of the triangle
$\xi\eta\om$,
we find that
$$2\le|v''w''|\le 2+3\de$$
for
$v''\in\xi\om$, $w''\in\xi\eta$
with
$b'(v'')=b'(w'')=(\eta|\om)_{b'}+1+\de$
for any Busemann function
$b'\in\cB(\xi)$.

Then by Lemma~\ref{lem:distort},
$|v'v''|\le\ln\frac{\sinh(2+3\de)}{2}$.
It remains to notice that
$|(\xi|\eta)_b-(\xi|\eta)_o|=|v\wt v|=|v'v''|$
because
$b(o)=0$
and
$(\xi|\eta)_b=b(\wt v)=|o\wt v|$.
\end{proof}

Recall that
$d_\infty$
is the spherical metric on
$\di M$
centered at
$o\in M$.

\begin{lem}\label{lem:smalldist} Given
$\xi$, $\eta\in\di M$
with
$d_\infty(\xi,\eta)\le 2e^{-(2+\de)}$,
then
$(\xi|\eta)_o\ge 1+\de$.
\end{lem}

\begin{proof} We let
$\xi_t\in o\xi$, $\eta_t\in o\eta$
with
$|o\xi_t|=|o\eta_t|=t$
for every
$t\ge 0$.
Recall that
$ds_t(v)=2\sinh t\cdot ds_\infty(v)$
for every
$v\in E\sub TU_oM$.
Thus
$|\xi_t\eta_t|\le d_t(\xi_t,\eta_t)\le e^td_\infty(\xi,\eta)$.
Using monotonicity of the Gromov product, we obtain
$$(\xi|\eta)_o\ge(\xi_t|\eta_t)_o=t-\frac{1}{2}|\xi_t\eta_t|
\ge t-e^{t-(2+\de)}=1+\de$$
for
$t=2+\de$.
\end{proof}

\begin{pro}\label{pro:main_2} For every
$\xi$, $\eta\in\di M$
with
$d_\infty(\xi,\eta)\le 2e^{-(2+\de)}$,
we have
$$d_\infty(\xi,\eta)\le c_2'e^{-(\xi|\eta)_o},$$
where
$c_2'=c_2e^{c_3}$, $c_2$
is the constant from Theorem~\ref{thm:main1}.
\end{pro}

\begin{proof} By Lemma~\ref{lem:smalldist}, we have
$(\xi|\eta)_o\ge 1+\de$.
Take
$\om\in\di M$
opposite to
$\xi$
with respect to
$o$
and
$b\in\cB(\om)$
with
$b(o)=0$.
Then by Lemma~\ref{lem:grprodiff},
$(\xi|\eta)_b\ge(\xi|\eta)_o-c_3$,
and using Theorem~\ref{thm:main1}, we obtain
$$d_\infty(\xi,\eta)\le d_b(\xi,\eta)\le c_2e^{-(\xi|\eta)_b}
\le c_2'e^{-(\xi|\eta)_o}.$$
\end{proof}

\begin{proof}[Proof of Theorem~\ref{thm:main2}] Let
$D=\diam\di M$
be the diameter of
$\di M$
with respect to (any) spherical metric
$d_\infty$
(for
$M=\K\hyp^n$,
this is clearly independent of
$\K=\C$, $\qH$, $\Ca$
and
$n\ge 2$,
cf. Lemma~\ref{lem:basic}).
We show that
$$d_\infty(\xi,\xi')\le c_2''e^{-(\xi|\xi')_o}$$
for each
$\xi$, $\xi'\in\di M$,
where
$c_2''=\max\{De^{1+\de},c_2'\}$.

If
$(\xi|\xi')_o\ge 1+\de$,
then as in Proposition~\ref{pro:main_2} we obtain
$d_\infty(\xi,\xi')\le c_2'e^{-(\xi|\xi')_o}$.
Thus we assume that
$(\xi|\xi')_o<1+\de$.
Then
$$d_\infty(\xi,\xi')\le D\le De^{1+\de}e^{-(\xi|\xi')_o}.$$
The estimate
$d_\infty(\xi,\xi')\ge c_1e^{-(\xi|\xi')_o}$
is obtained in Lemma~\ref{lem:below_sph}.
\end{proof}

\end{document}